\documentclass[12pt]{article}

\usepackage{latexsym}
\usepackage[cm]{fullpage}
\usepackage{verbatim}
\usepackage{amsthm}
\usepackage{amsfonts}
\usepackage{amsmath}
\usepackage{amssymb}
\usepackage{amscd}

\newlength{\abstractwidth}
\renewcommand{\thefootnote}{\fnsymbol{footnote}}
\renewcommand{\thanks}[1]{\footnote{#1}} 
\newcommand{\starttext}{ \setcounter{footnote}{0}
\renewcommand{\thefootnote}{\arabic{footnote}}}

\newcommand{\be}{\begin{equation}}
\newcommand{\bea}{\begin{eqnarray}}
\newcommand{\eea}{\end{eqnarray}} 
\newcommand{\ee}{\end{equation}}
 \newcommand{\<}{\langle}
\renewcommand{\>}{\rangle}
\def\ba{\begin{eqnarray}}
\def\ea{\end{eqnarray}}

\numberwithin{equation}{section}

\def\C{\mathbb C^n}
\newtheorem{prop}{Proposition}[section]
\newtheorem{theorem}[prop]{Theorem}
\newtheorem{lemma}[prop]{Lemma}
\newtheorem{corollary}[prop]{Corollary}
\newtheorem{remark}[prop]{Remark}
\newtheorem{example}[prop]{Example}

\newtheorem{question}[prop]{Question}

\def\begeq{\begin{equation}}
\def\endeq{\end{equation}}

\def\<{\langle}
\def\>{\rangle}
\def\({\left(}
\def\){\right)}
\def\pa{\partial}

\def\p{\partial}
\def\pa{\partial}

\def\om{\omega}
\def\Om{\Omega}
\def\we{\wedge}

\def\be{\beta}
\def\ep{\epsilon}
\def\dc{dd^c}
\def\ddb{\partial\bar\partial}

\def\n{\nabla}

\DeclareMathOperator{\ksh}{\it \mathcal {SH}_k\rm}

 \DeclareMathOperator{\la}{\lambda}
\DeclareMathOperator{\co}{\mathbb C}

\begin{document}

\starttext \baselineskip=15pt \setcounter{footnote}{0}

\begin{center}
{\Large\bf
Regularity of Degenerate Hessian Equations}
\end{center}

\medskip

\centerline{S\l awomir Dinew, Szymon Pli\'s, and Xiangwen Zhang\footnote{The first and second named authors
were supported by the NCN grant 2013/08/A/ST1/00312. The third named author was supported by the Simons Collaboration
Grant-523313.}}

\bigskip

\begin{abstract}
{\small We show a second order a priori estimate for solutions to the complex $k$-Hessian equation on a compact K\"ahler manifold provided the $(k$-$1)$-st root of the right hand side is $\mathcal C^{1,1}$. This improves an estimate of Hou-Ma-Wu \cite{HMW}. An example is provided to show that the exponent is sharp.}
\end{abstract}


%
%

\section{Introduction}
\par
Geometrically motivated complex fully nonlinear elliptic partial differential equations have received a lot of attention 
recently (see \cite{PPZ, PPZ1, PT, Sze1, TW} which is by far an incomplete list of recent important contributions).
The solvability of such equations is usually studied through the continuity method and boils down to establishing a priori estimates just as in the classical approach of S. T. Yau \cite{Y}. 

In general the considered problems are reducible to a scalar equation satisfied by a real valued function $u$ defined of a compact complex manifold $X$ equipped with a fixed   Hermitian form $\om$. Quite often additonal assumptions such as k\"ahlerness of $\om$ are imposed and then the real $(1,1)$-form
$\omega+i\pa\bar{\pa}u$ is the geometric object with the desired properties. Arguably the most natural geometric assumption is that  $\omega+i\pa\bar{\pa}u$ defines a {\it metric} i.e. it is positive definite.
However, it often happens that the very nature of the nonlinearity imposes more general {\it admissibility} conditions (see for example \cite{Hou, HMW, PPZ, Sze1}). This lack of positivity usually contributes significantly to the technical difficulty of the estimates.

\smallskip

In this note we deal with the {\it complex Hessian equations} on a compact K\"ahler manifold $(X,\om)$ with $dim_{\co}X=n$. These interpolate between the Laplace equation (in the case $k=1$) and the Monge-Amp\`ere equation in the case $k=n$. They are defined by 

\begin{equation}\label{sigmak}
(\om+i\ddb u)^k\we\om^{n-k}=f\om^n,
\end{equation}
where the given nonnegative function $f$ satisfies the necessary compatibility condition 
$\int_X\om^n=\int_Xf\om^n.$

For smooth $u$ the {\it admissibility} condition imposed of the class of solutions $u$ is that
\[
(\om+i\ddb u)^j\we\om^{n-j}\geq 0,\ j=1,2,\cdots, k.
\]
We denote the class of such functions by $\ksh(X,\om)$. Note also that adding a constant to a solution $u$ doesn't change the equation
(\ref{sigmak}), thus we normalize the solutions by imposing the condition $\int_Xu\,\om^n=0$.
The solvability of the equation (\ref{sigmak}) was established for smooth strictly positive right hand side data $f$ satisfying the compatiblity condition through the works of Hou-Ma-Wu \cite{HMW} who proved the uniform and second order a priori estimates and the first named author and Ko\l odziej \cite{DK} who obtained the missing gradient estimate by an indirect blow-up argument. 

\smallskip

Having the existence of smooth solutions for smooth strictly positive data it is natural to address the regularity theory in the degenerate cases.  A situation of special interest is when the right hand side function is allowed to vanish.  Such a scenario, reminiscent of failure of strict ellipticity in linear PDEs, as a rule implies the occurence of {\it singular solutions}. In view of the classical theory in the Monge-Amp\`ere case (see \cite{B3, Gu}) the maximum one can expect in this setting is $\mathcal C^{1,1}$ regularity.

A natural question appears then about optimal conditions implying that $u\in \mathcal C^{1,1}$. Note that in \cite{HMW} the authors have proven that the complex Hessian is controlled by the gradient of $u$ provided the $\mathcal C^2$ norm of $f^{1/ k}$ is under control. This may hold even if  $f$ vanishes somewhere, that is, we deal with the degenerate equation. The estimate in \cite{HMW} left the problem whether the exponent $1/k$ is the optimal one. We will show that one can further improve it to $1/(k-1)$ if $k\geq 2$ as one expects for the real case (for $k=1$ we have the Poisson equation whose regularity theory is classical). Our main result is the following:


\begin{theorem} 
Let $f\geq 0$ be a function on compact K\"ahler manifold $(X, \omega)$ satisfying $\int_Xf\om^n=\int_X\om^n$. Assume that $f^{1/(k-1)}\in \mathcal C^{1,1}$. Then the solution $u$ to the equation (\ref{sigmak}) admits an a priori estimate 
\vspace{-0.3cm}
\begin{equation}
\vspace{-0.3cm}
\sup_{X} \|i\ddb u\|\leq C(1+\|Du\|^2)
\end{equation}
for some uniform constant $C$ dependents only on $n, k, X, \|f^{1/(k-1)}\|_{\mathcal C^{1, 1}}$, the oscillation $osc_Xu$ of $u$ and the lower bound on the bisectional curvature of $\om$. 
\end{theorem}

Coupling this estimate with the main result from \cite{DK} one can prove that the solution $u$ has bounded Laplacian and thus  belongs to the {\it weak} $\mathcal C^{1,1}$ space.
The proof of the above estimate relies heavily on the argument of Hou-Ma-Wu \cite{HMW}. The main importance of our improvement is that the obtained exponent is {\it optimal} as an example constructed in the note shows.

\medskip

The complex Hessian equations were first considered in the case of domains in $\co^n$, where the equation takes the form
\vspace{-0.2cm}
\begin{equation}\label{local}
 (i\ddb u)^k\we\beta^{n-k}=f\beta^n
\end{equation}
with $\beta=\dc |z|^2$ denoting the standard Hermitian $(1,1)$ form in $\co^n$. The corresponding Dirichlet problem was studied by S.-Y. Li \cite{L} and Z. B\l ocki \cite{B2}. In particular, the nondegenerate Dirichlet problem in a {\it strictly k-pseudoconvex} domain admits a unique smooth solution provided $f$ and the boundary data are smooth and $f$ is uniformly positive. Again it is interesting to study $\mathcal C^{1,1}$
regularity in the case when $f$ vanishes or decreases to zero at the boundary. It has to be emphasized that the occurence of a boundary makes things substantially harder and the regularity theory is far from complete. When $k=n$, that is for the complex Monge-Amp\`ere case, some regularity results were obtained by Krylov \cite{Kry1, Kry2} under the assumption that $f\geq 0$ and $f^{1/k}\in \mathcal C^{1, 1}$.

\medskip

The complex Hessian equation is itself modelled on its {\it real} counterpart 
\[
S_k(D^2u)=f
\]
with $S_k(A)$ denoting the sum of all main $k\times k$-minors of the matrix $A$. The real Hessian equation is much better understood and we refer to \cite{W2} for an excellent survey regarding the
corresponding regularity theory. In particular in the real setting the following analogue of the Hou-Ma-Wu \cite{HMW} estimate was established by Ivochkina-Trudinger-Wang in \cite{ITW}.
\begin{theorem}[\cite{ITW}]Let $U\subset\mathbb R^n$ be a strictly $k$-convex domain with $ C^4$ boundary. Suppose that the admissible function $v$ 
satisfies the problem
\begin{equation}\label{real}
 \begin{cases}
 S_k(D^2v)=f\ & {\rm in}\ U\\
 v=\varphi\ & {\rm in}\ \pa U, 
\end{cases}
\end{equation}
where we assume that $\varphi\in \mathcal C^4(\pa U)$ and $f^{1/k}\in \mathcal C^2(\overline{U})$. Then $v\in \mathcal C^{1,1}(\overline{U})$ with $\mathcal C^2$
norm bounded by an estimable constant dependent on $f,\varphi, k, n$ and $U$.
\end{theorem}

Again it is unknown whether the exponent $1/k$ is optimal here. It has attracted much attention to establish the above theorem with the exponent $1/k$ being replaced by $1/(k-1)$. More recently, the above theorem was proved under a weaker condition on $f$ in \cite{WX}, but the optimal one seems still missing. On the bright side the optimality problem has been settled in the extremal case $k=n$, i.e. when we deal with the real Monge-Amp\`ere equation. By a result of Guan-Trudinger-Wang \cite{GTW} the {\it optimal} exponent yielding $\mathcal C^{1,1}$ solutions is $1/(n-1)$ for domains in $\mathbb R^n$. Sharpness of this bound follows from an example of Wang \cite{W1}. This example has been generalized for the complex Monge-Amp\`ere equation by the second named author in \cite{P}.

\smallskip

In the case of general Hessian equations the current state of affairs is as follows: it was stated in \cite{ITW} that an example analogous to the one in \cite{W1} suggests that the exponent $1/(k-1)$ is optimal for the real $k$-Hessian equation. As no proof of this was provided we take the opportunity to present the relevant example (as well as its complex and compact manifold counterparts) in detail, since in our opinion the arguments used in the proof have to be slightly different than the approach of Wang \cite{W1}. In particular, we have

\begin{prop}\label{Plis} 
For every $\varepsilon>0$, there exists a non-negative function $f$ in the unit ball in $\mathbb R^n$ (respectively, $\mathbb P^{n-1}\times \mathbb P^1$ or the unit ball in\ $\co^n$) such that 
$f^{1/(k-1)+\varepsilon}\in \mathcal C^{1,1}$, but the solution to the k-Hessian equation with $f$ 
as a right hand side is not $\mathcal C^{1,1}$. 
\end{prop}

The examples living on  $\mathbb P^{n-1}\times\mathbb P^1$ equipped with the Fubini-Study product metric yield in particular a regularity threshold $1/(k-1)$ for the exponent of $f$. This shows that our main result (Theorem 1.1) is {\it optimal}. We also take the opportunity to investigate the regularity
of the example given in Proposition \ref{Plis} under various weaker assumptions on the right hand side (see Example \ref{general-example} in \S 4.2). More precisely, we provide some examples to indicate what might be the best possible regularity of the admissible solutions for equation
\[
(\omega+ i\ddb u)^k\wedge\omega^{n-k} = f(z)\, \omega^n,
\] 
on a compact K\"ahler manifold $(X, \omega)$ with $0 \leq f\in L^{p}$ (or $\mathcal C^{0, \delta}$) satisfying $\int_X f\, \omega^n = \int_X \omega$. 
We believe that at least in some cases the obtained examples are sharp.

\bigskip

{\bf Acknowledgements}. This project was initiated when the first named author was visiting University of California at Irvine in the summer of 2016. He  wishes to thank the Department of Mathematics for the warm hospitality.

\

\section{Preliminaries}

\par
Below we gather the definitions and facts that will be used in the proofs later on. We refer to the survey article \cite{W2} for the basics of the theory of Hessian equations. We start with some relevant notions from linear algebra. Consider the set $\mathcal M_n(\mathbb R)$ (respectively: $\mathcal M_n(\mathbb C)$) of all symmetric (respectively Hermitian symmetric) $n\times n$ matrices. Let $ \la(M) = (\la _1, \la _2 , ... ,
\la _n)$ be the eigenvalues of a matrix $M$ arranged in decreasing order and
let
\[
S_k(M)=S_k(\la(M))=\sum _{0<j_1 < ... < j_m \leq n }\la _{j_1}\la _{j_2} 
... \la _{j_m}
\]
be the $k$-th elementary symmetric polynomial applied to the vector $\la(M)$. Analogously we define $\sigma_k(M)$ if $M$ is Hermitian. Then one can define the positive cones $\Gamma_m$ as follows
\begin{equation}\label{ga}
\Gamma_m=\lbrace \la\in\mathbb R^n|\ S_1(\la)> 0,\ \cdots,\
S_m(\la)> 0\rbrace.
\end{equation}
Note that the definition of $\Gamma_m$ is nonlinear if $m>1$.

\smallskip

Let now $V=(v_{\bar{k} j})$ be a fixed positive definite Hermitian matrix and
$\la_i(T)$ be the eigenvalues of a Hermitian matrix $T=(\tau_{\bar{k}j})$ with
respect to $V$. We can define analogously $\sigma_{k,V}(T)$. In the language of differential forms if 
$\tau=i \,\tau_{\bar{k}j}dz^j\we d\bar{z}^k$, $v=i\,v_{\bar{k}j}dz^j\we d\bar{z}^k$
then $\sigma_{k,V}(T)$ is (up to a multiplicative universal constant) equal to the coefficient of the top-degree form $\tau^{k}\we v^{n-k}$. We can also analogously define the sets $\Gamma_k(V)$. Below we list the properties of these cones that will be used later on:
\begin{enumerate}
\item (Maclaurin's inequality I) If $\la\in\Gamma_m$ then
$\left(\frac{S_j}{\binom nj}\right)^{\frac 1j}\geq \left(\frac{S_i}{\binom ni}\right)^{\frac 1i}$
for $1\leq j\leq i\leq m$. The same inequality holds for the operators $\sigma_k$;
\item (Maclaurin's inequality II) There is a universal constant $c(n,m)$, dependent only on $n$ and $m$, such that
$\sigma_{m-1}(\la)\geq c(n,m)\sigma_m(\la)^{\frac{m-2}{m-1}}\sigma_1(\la)^{\frac1{m-1}}$ for any $\la\in \Gamma_m$;
\item $\Gamma_m$ is a convex cone for
any $ m$ and the function $\sigma_m^{\frac 1m}$ as well as $log(\sigma_m)$ are concave when restricted to
$\Gamma_m$;
\item (G\aa rding's inequality) Let $\sigma_{k}(\la|i):=
\frac{\partial \sigma_{k+1}}{\partial \la_i}(\la)$. Then for any
$\la,\ \mu\in\Gamma_m$
$$\sum_{i=1}^n\mu_i\sigma_{m-1}(\la|i)\geq m\sigma_m(\mu)^{\frac1m}\sigma_m(\la)^{\frac{m-1}{m}}.$$
\item $\sigma_{m-1}(\la|ij)=\frac{\sigma_m(\la|i)-\sigma_m(\la|j)}{\la_j-\la_i}$ for all $i\neq j$.
\end{enumerate}

We refer to \cite{W2} for further properties of these cones.

\medskip

Recall that a smooth function $v$ living on a domain $U\subset \mathbb R^n$ is called $k$-convex for some natural $1\leq k\leq n$ if 
$$S_j(D^2v(x))\geq0, j=1,\cdots, k$$
with $D^2v(x)$ denoting the  Hessian matrix of $v$ at $x$ and $S_j(A)$ is the sum of all main $j\times j$ minors of the $n\times n$ matrix $A$. Analogously a function $u$ living on a domain $\Omega\subset \mathbb C^n$ is called $k$-subharmonic for some natural $1\leq k\leq n$ if 
$$\sigma_j(i\partial\bar\partial u(z))\geq0, j=1,\cdots, k$$
with $\sigma_j(B)$ denoting again the sum of the main $j\times j$ minors of a Hermitian symmetric matrix $B$. In the complex setting one can alternatively use the laguage of differential forms to define the $\sigma_k$ operator as
\[
\sigma_k(i\partial\bar \partial u) \beta^n=\binom{n}{k}(i\ddb u)^k\wedge\beta^{n-k}
\]
with $\beta:=\dc |z|^2$ denoting the standard Hermitian $(1,1)$-form in $\mathbb C^n$.

\smallskip
These are the local real and complex versions of the functions belonging to $\ksh(X,\om)$ defined in the introduction. 
In each of these settings one can define {\it singular} $k$-convex (respectively $k$-subharmonic) functions locally 
as decreasing limits of smooth ones. The basic fact from the associated nonlinear potential theories (see \cite{W2}
for the real case and \cite{B2, DK1} for the complex one) is that the operators $S_k$ (respectively $\sigma_k$) 
can still be properly defined as nonnegative measures for singular bounded $k$-convex ($k$-subharmonic) functions.

The following theorem, known as the comparison principle is basic in the potential theory of $k$-subharmonic functions (see \cite{DK1}). We remark that an analogous result is also true for $k$-convex functions (\cite{W2}).

\begin{theorem}[\cite{W2}]\label{compprin}
If $u,\ w$ are two bounded $k$-subharmonic functions in a domain $\Om\subset\mathbb C^n$, such that $\liminf_{z\rightarrow\partial\Om}(u-w)(z)\geq 0$. If moreover
$$\sigma_k(i\partial\bar \partial w)\geq\ \sigma_k(i\partial\bar \partial  u)$$
as measures then $u\geq w$ in $\Om$. 
\end{theorem}

As a corollary one immediately obtains the uniqueness of bounded solutions for the corresponding Dirichlet problems.
The uniqueness of normalized bounded solutions from $\ksh(X,\om)$ is also true (see \cite{DC}). The corresponding comparison 
principle (see \cite{DK}) reads as follows:
\begin{theorem}\label{compprinkahler}
Let $\varphi,\ \psi\in \ksh(X,\om)$ be bounded. Then
\[
\int_{\{ \varphi<\psi\}}(\om+i\ddb\psi)^k\wedge\om^{n-k}\leq\int_{\{ \varphi<\psi\}}(\om+i\ddb\varphi)^k\wedge\om^{n-k}.
\]
\end{theorem}

Finally we shall need an elementary calculus lemma whose proof can be found in \cite{B1}:
\begin{lemma} If $\psi\in \mathcal C^{1,1}(\overline{\Omega})$ is a nonnegative function. Then $\sqrt{\psi}$ is 
locally Lipschitz  in $\Omega$. For almost every $x\in\Omega$ we have
\[
\left|D\sqrt{\psi}(x)\right|\leq \max\left\{\frac{|D\psi(x)|}{2dist(x,\pa\Omega)}, \frac{1+\sup_{\Om}\la_{\max}[D^2\psi]}{2}\right\},
\]
where $\la_{\max}[D^2\psi]$ denotes the maximum eigenvalue of the real Hessian of $\psi$.
\end{lemma}
Working in charts on a compact K\"ahler manifold one easily gets the following corollary of the lemma above:
\begin{corollary}\label{blocki}
Let $f\geq 0$ be a function on a compact K\"ahler manifold $(X,\om)$ such that $f^{1/(k-1)}\in \mathcal C^{1,1}(X)$. Then
\[
\left\|\nabla f^{1/(k-1)}(z)\right\|^2\leq C\left\|f^{1/(k-1)}(z)\right\|
\]
for some constant $C$ dependent on $X,\om$ and the $\mathcal C^{1,1}$ norm of $f^{1/(k-1)}$. In particular for any unitary vector $\eta$ one has
\[
\partial_{\eta}\partial_{\bar\eta}\log f= (k-1) \left({\partial_{\eta}\partial_{\bar\eta} f^{1/(k-1)} \over f^{1/(k-1)}} - { |\partial_{\eta} f^{1/(k-1)}|^2 \over f^{2/(k-1)}}\right) \geq - {\tilde{C} \over f^{1/(k-1)}} 
\]
for some constant $\tilde{C}$ dependent on $X,\om$ and the $\mathcal C^{1,1}$ norm of $f^{1/(k-1)}$.
\end{corollary}
\begin{proof}
 Pick a point $z\in X$ and a chart centered around $z$ containing a ball of some fixed radius $r$ (dependent only on $X$ and $\om$). Then we apply the 
 lemma for $\psi=f^{1/(k-1)}$ in the coordinate ball centered at $z$ with radius $r$ to get the statement.
\end{proof}


\begin{remark} In the corollary it is crucial that the manifold has no boundary. As observed in \cite{B1} the function 
$\psi(t)=t$ on $(0,1)$ shows that it is in general impossible to control $|D\psi|^2$ by $\psi$ globally in the presence of boundary.
\end{remark}

\noindent{\bf Notation}. Throughout the paper, $(X,\om)$ will denote a compact K\"ahler manifold, $\Om$ will be a domain 
in $\co^n$ and $U$ will be a domain in $\mathbb R^n$ for some $n\geq 2$. The constant $C_0$ denotes the lower bound for the bisectional curvature associated to $\om$ i.e.
\begin{equation}\label{Czero}
C_0:=\sup_{x\in M}|\inf_{\eta,\zeta}R_{\eta\bar{\eta}\zeta\bar{\zeta}}|
\end{equation}
with $\zeta,\ \eta$ varying among the unit vectors in $T_xX$. Other constants dependent only on the pertinent quantities will be denoted by $C, C_i$ or $c_i$. We shall refer to these as constants {\it under control}.

 \section{The main estimate}
This section is devoted to the proof of the following a priori estimate:
\begin{theorem}\label{aaa} Let $u\in\ksh(X,\om)$ be a $\mathcal C^4(X)$ function solving the problem
\begin{equation}\label{dirichlet}
\begin{cases} 
(\om+i\ddb u)^k\wedge\om^{n-k}=f\om^n\\
\int_Xu\,\om^n=0
\end{cases}
\end{equation}
where the nonnegative function $f$ satisfies the compatibility assumption $\int_Mf\,\om^n=\int_M\om^n$.
Suppose that $\|f\|_{\mathcal C^0}\leq B, \, \|f^{1/(k-1)}\|_{\mathcal C^1}\leq B$ and $\|f^{1/(k-1)}\|_{\mathcal C^{2}}\leq B$.
Then
\begin{equation}\label{main-est}
\sup_M\|i\ddb u\|_{\om}\leq C(\sup_M\|\n u\|^2+1)
\end{equation}
 for some constant $C$ dependent on $C_0,\ B,\ \om, n$ and $k$.
 \end{theorem}
 
\noindent Using the above $\mathcal C^2$ estimate, one can repeat the blow-up argument from \cite{DK} to deduce an indirect gradient bound for $u$. Coupling this information with (\ref{main-est}), we get the following result:
\begin{theorem}\label{bbb} If $u\in\ksh(X,\om)$ solves the problem (\ref{dirichlet}) with the assumption $f^{1/(k-1)}\in\mathcal C^{1,1}$, then $u$ belongs to the weak $\mathcal C^{1,1}$ space, i.e. the Laplacian of $u$ is bounded.
\end{theorem}
\begin{proof}[Proof of Theorem \ref{bbb}]
The argument can be found in \cite{B1}. We provide the details for the sake of completeness.

Given any $f$ as in the statement there is a family $f_{\ep}, \ep\in (0,1)$ of smooth strictly positive functions uniformly  convergent as $\ep\searrow 0$ to $f$ such that additionally  $f_\ep^{1/(k-1)}$ tends to $f^{1/(k-1)}$ in $\mathcal C^{1,1}$ norm (one way to produce such a family is to use a convolution in local charts coupled with a partition of unity, see \cite{B1} for the details). Let also $C_{\ep}$ be a positive constant such that
$$\int_MC_\ep f_{\ep}\om^n=\int_Mf\om^n=\int_M\om^n.$$
It follows that $\lim_{\ep\rightarrow 0}C_\ep=1$. Furthermore we can assume that 
\[
\|(C_\ep f_\ep)^{1/(k-1)}\|_{\mathcal C^2}\leq 2\|f^{1/(k-1)}\|_{\mathcal C^2}.
\]

\noindent Hence the solutions $u_\ep\in\ksh(X,\om)$ to the problem
\begin{equation}\label{dirichletep}
\begin{cases} 
(\om+i\ddb u_\ep)^k\wedge\om^{n-k}=C_\ep f_\ep\om^n\\
\int_Xu_\ep\,\om^n=0,
\end{cases}
\end{equation}
(which are smooth by the Calabi-Yau type theorem from \cite{DK}) converge in $L^1(X,\om)$ to $u$ (see Corollary 4.2 in \cite{C}). 

On the other hand we have as an application of Theorem \ref{aaa} the bound
$\Delta_{\om}u_\ep\leq C$
for a constant $C$ dependent only on $\|f^{1/(k-1)}\|_{\mathcal C^2}, n, k$ and the lower bound of the bisectional curvature. In particular the bound does not depend on $\ep$ and hence passing to the limit we obtain $\Delta_{\om}u\leq C$ which implies the claimed result.
\end{proof}

\

\noindent Now we proceed to the proof of the main a priori estimate:
\begin{proof}[Proof of Theorem \ref{aaa}]
We will work, just as in \cite{B1} under the assumption that $f>0$ (for example using  the approximate problems (\ref{dirichletep})) and we will obtain an estimate independent of $\inf_Xf$. This is done in order to avoid confusion as we shall divide by $f$ in the argument. Then, if needed, one can repeat the final part of the argument in the proof of Theorem \ref{bbb} to drop the assumption $f>0$. Our proof will follow closely the argument in \cite{HMW}. 
\smallskip

Given a point $x\in M$ we consider a fixed local coordinate system $(z_1,\cdots, z_n)$ centered at $x$. By re-choosing the coordinates if necessary, one can assume that the form $\om=i \, g_{\bar{k}j}dz^j\we d\bar{z}^k$ is diagonal at $x$.
We follow the notation in \cite{HMW} and use the covariant derivatives with respect to the background K\"ahler metric $\omega$ to do the calculation. In particular for any function $h$ defined near $x$ let 
$h_i=\n_{\p/\p z^{i}} h,$ $h_{i\bar{j}}=\n_{\p/\p \bar{z}^{j}}\n_{\p/\p z^{i}}h,$ etc. 

\smallskip

As in \cite{HMW}, we consider the quantity
\begin{equation}\label{G}
\tilde{G}(z,\xi):=\log(1+u_{i\bar{j}}\xi^{i}\bar{\xi}^j)+\varphi(|\n u|^2)+\psi(u)
\end{equation}
defined for any $z\in X$ and any unit vector $\xi\in T_z^{1,0}X$. The relevant quantities are defined as follows:
\begin{equation}\label{phi}
\varphi(t):=-{1\over 2}\log(1-{t\over 2K}) \ \  \textit{ \rm with } K:=\sup_M|\n u|^2+1;
\end{equation}
and
\begin{equation}\label{psi}
\psi(t):=-A\log(1+{t\over 2L}) \ \ \textit{ \rm with } L:=\sup_M|u|+1,\ \ A=3L(2C_0+1).
\end{equation}
\smallskip

The  properties of $\varphi$ and $\psi$ that we shall use are listed below:
\begin{equation}\label{propphi}
\frac12\log2\geq\varphi(|\n u|^2)\geq 0,\ \ \ \frac1{2K}\geq\varphi'(|\n u|^2)\geq \frac1{4K}>0,
\end{equation}
\begin{equation}\label{varphi12}
 \varphi''(|\n u|^2)=2[\varphi'(|\n u|^2)]^2>0
\end{equation}
and
\begin{equation}\label{proppsi}
A\log{1\over 2}\geq\psi\geq  A\log{2\over 3},\ \ \  \frac AL\geq-\psi'(u)\geq\frac A{3L}=2C_0+1,
\end{equation}
\begin{equation*}
\psi''(u)\geq\frac{2\ep}{1-\ep}(\psi'(u))^2, \ {\rm for\ all}\ \ep\leq\frac1{2A+1}.
\end{equation*}
Suppose $\tilde{G}$ attains maximum a point $x_0\in X$ and a tangent direction $\xi_0\in T_{x_0} X$. In a standard way we construct normal coordinate system at $x_0$ and assume that $\xi_0=g_{1\bar{1}}^{-1/2}\frac{\p}{\p z^1}$. We may also assume that $u_{i\bar{j}}$ is diagonal at $x_0$, i.e.,
\[
u_{i\bar{j}}(x_0)=\delta_{ij}u_{i\bar{i}}(x_0).
\]
Then $\la_i:=1+u_{i\bar{i}}(x_0)$ are the eigenvalues of $\om+dd^cu$ with respect to $\omega$ at $x_0$. Therefore, near $x_0$, the function 
\begin{equation}
G(z)=\log(1+g_{1\bar{1}}^{-1}u_{1\bar{1}})+\varphi(|\n u|^2)+\psi(u)
\end{equation}
is well defined and has a maximum at $x_0$. At this moment we mention that $u_{1\bar{1}}(x_0)$ is of the same size as $\Delta_{\om}u(x_0)$ (meaning that for a numerical constant $C_n$ one has $C_n^{-1}u_{1\bar{1}}(x_0)\leq \Delta_{\om}u(x_0)\leq C_nu_{1\bar{1}}(x_0)$), since $\la_i\in \Gamma_k$ with $k\geq 2$ and hence
$\sum_{j=2}^n\la_i\geq 0.$
Note that in order to get the claimed global  bound for the Laplacian in terms of the supremum of the gradient it is thus  sufficient to bound $u_{1\bar{1}}(x_0)$ by an expression which is of linear growth in $K$. To this end let us take the nonlinear operator
\[
S(\om+i\ddb u):=\log\sigma_k(\omega+i\ddb u)
\]
which is different from $F=\sigma_k^{1/k}$ used in \cite{HMW}. Using the diagonality of $\om$ and $u_{i\bar{j}}$  at $x_0$ we compute that
\begin{equation}\label{sij}
S^{i\bar{j}}:=\frac{\p S(\om+i\ddb u)}{\p u_{i\bar{j}}}=\delta_{ij}\frac{\sigma_{k-1}(\la|i)}{\sigma_k(\la)}.
\end{equation}
At $x_0$ the second derivatives $S^{i\bar{j},p\bar{q}}:=\frac{\p^2 S}{\p u_{i\bar{j}}\p u_{p\bar{q}}}$ are zero except in the following cases:
\[
S^{i\bar{i},p\bar{p}}=(1-\delta_{ip})\frac{\sigma_{k-2}(\la|ip)}{\sigma_k(\la)}-\frac{\sigma_{k-1}(\la|i)\sigma_{k-1}(\la|p)}{\sigma_k^2(\la)}
\]
and for $i\neq p$
\[
S^{i\bar{p},p\bar{i}}=-\frac{\sigma_{k-2}(\la|ip)}{\sigma_k(\la)}.
\]
\par
Observe also that at $x_0$
\begin{equation}\label{Sandu}
\sum_{i=1}^nS^{i\bar{i}}(1+u_{i\bar{i}})=\sum_{i=1}^nS^{i\bar{i}}\la_i=k.
\end{equation}

\noindent Differentiating the equation $S(\omega+dd^cu)=\log f$ and commuting the covariant derivatives we obtain the formulas (compare \cite{HMW}) that at $x_0$
\begin{equation}\label{commute3}
\sum_{p=1}^nS^{p\bar{p}}u_{jp\bar{p}}=(\log f)_j+\sum_{p,q=1}^nu_qS^{p\bar{p}}R_{j\bar{q}p\bar{p}}
\end{equation}
and
\begin{equation}\label{commute4}
\sum_{p=1}^nS^{p\bar{p}}u_{1\bar{1}p\bar{p}}=(\log f)_{1\bar{1}}-\sum_{i,j,r,q=1}^nS^{i\bar{j},r\bar{q}}u_{i\bar{j}1}u_{r\bar{q}\bar{1}}+\sum_{p=1}^n
S^{p\bar{p}}(u_{1\bar{1}}-u_{p\bar{p}})R_{1\bar{1}p\bar{p}}.
\end{equation}
Returning to $G$ from the extremal property at $x_0$ we have the following formula
\begin{equation}\label{firstorder}
0=G_p=\frac{u_{1\bar{1}p}}{1+u_{1\bar{1}}}+\varphi'u_pu_{\bar{p}p}+\varphi'\sum_{j=1}^nu_{jp}u_{\bar{p}}+\psi'u_p.
\end{equation}
Also by diagonality, ellipticity, the equation itself and (\ref{commute4}) we get
\begin{eqnarray}\label{secondorder}
0&\geq &\sum_{p=1}^nS^{p\bar{p}}G_{p\bar{p}}\\
&=&\sum_{p=1}^n\frac{S^{p\bar{p}}u_{1\bar{1}p\bar{p}}}{1+u_{1\bar{1}}}-\sum_{p=1}^n\frac{S^{p\bar{p}}|u_{1\bar{1}p}|^2}{(1+u_{1\bar{1}})^2}+2 \varphi' {\rm Re}[(\log f)_{\bar{j}}u_j]+\varphi'\sum_{p,q,r=1}^nu_{\bar{r}}u_qS^{p\bar{p}}R_{p\bar{p}r\bar{q}}\nonumber\\
&&+\sum_{p=1}^n\varphi'S^{p\bar{p}}|u_{p\bar{p}}|^2+\sum_{p=1}^n\varphi'S^{p\bar{p}}\sum_{j=1}^n|u_{jp}|^2+\varphi''\sum_{p=1}^nS^{p\bar{p}}|\sum_{j=1}^nu_{jp}u_{\bar{j}}+u_pu_{p\bar{p}}|^2\nonumber\\
&&+\psi''\sum_{p=1}^nS^{p\bar{p}}|u_p|^2+\psi'k-\psi'\sum_{p=1}^nS^{p\bar{p}}.\nonumber
\end{eqnarray}
The first term can be estimated by exploiting (\ref{commute4}), analogously to \cite{HMW} we have
\begin{eqnarray}\label{4orderterm}
\sum_{p=1}^n\frac{S^{p\bar{p}}u_{1\bar{1}p\bar{p}}}{1+u_{1\bar{1}}}\geq 
-\la_1^{-1}\sum_{i,j,r,q=1}^nS^{i\bar{j},r\bar{q}}u_{i\bar{j}1}u_{r\bar{q}\bar{1}}-C_0\sum_{p=1}^nS^{p\bar{p}}-C_0k+\frac{(\log f)_{1\bar{1}}}{\la_1}.
\end{eqnarray}
Denote $\mathcal S:=\sum_{p=1}^nS^{p\bar{p}}$. Then the fourth term in (\ref{secondorder}) can be estimated from below by
\begin{equation}\label{curvterm}
\varphi'\sum_{p,q,r=1}^nu_{\bar{r}}u_qS^{p\bar{p}}R_{p\bar{p}r\bar{q}}\geq -K\varphi'\mathcal S\, C_0\geq -\frac{C_0}2\mathcal S,
\end{equation}
where we used the property (\ref{propphi}) of $\varphi'$. The fifth term can be rewritten as
\begin{equation}\label{laisquare}
\sum_{p=1}^n\varphi'S^{p\bar{p}}|u_{p\bar{p}}|^2=\sum_{p=1}^n\varphi'S^{p\bar{p}}|\la_p-1|^2=\sum_{p=1}^n\varphi'S^{p\bar{p}}\la_p^2-2\varphi' k+\varphi' \mathcal S.
\end{equation}
The sixth term is obviously nonnegative. So coupling (\ref{secondorder}) with (\ref{4orderterm}), (\ref{curvterm}) and (\ref{laisquare}) we obtain
\begin{eqnarray}\label{newsecondorder}
0&\geq& -\sum_{i,j,r,q=1}^n\frac{S^{i\bar{j},r\bar{q}}u_{i\bar{j}1}u_{r\bar{q}\bar{1}}}{1+u_{1\bar{1}}}-\sum_{p=1}^n\frac{S^{p\bar{p}}|u_{1\bar{1}p}|^2}{(1+u_{1\bar{1}})^2}\\\nonumber
&&+\psi''\sum_{p=1}^nS^{p\bar{p}}|u_p|^2+\varphi''\sum_{p=1}^nS^{p\bar{p}}|\sum_{j=1}^nu_{jp}u_{\bar{j}}+u_pu_{p\bar{p}}|^2+\varphi'\sum_{p=1}^nS^{p\bar{p}}\la_p^2\\\nonumber
&&+(-\psi'+\varphi'-2C_0)\mathcal S+\frac{(\log f)_{1\bar{1}}}{\la_1}+2\varphi' {\rm Re}[(\log f)_{\bar{j}}u_j]-(2\varphi'+\psi'-C_0 )k.
\end{eqnarray}

\medskip
Up to now we have followed \cite{HMW}. The big difference is that the last three terms, contained in the constant $C_2$ in \cite{HMW}, are not controllable from below in our setting. Define the constant $\delta:=\frac{1}{2A+1}$. Let us divide the analysis into two separate cases:

\smallskip

{\bf Case 1: Suppose that $\la_n<-\delta\la_1$}. Using the critical equation (\ref{firstorder}), we can exchange the term second term in (\ref{newsecondorder}) by
\[
-\sum_{p=1}^n\frac{S^{p\bar{p}}|u_{1\bar{1}p}|^2}{(1+u_{1\bar{1}})^2} =-\sum_{p=1}^nS^{p\bar{p}}|\varphi'u_pu_{\bar{p}p}+\varphi'\sum_{j=1}^nu_{jp}u_{\bar{p}}+\psi'u_p|^2.
\]
By Schwarz inequality this is further estimated from below by
\[
-\sum_{p=1}^n\frac{S^{p\bar{p}}|u_{1\bar{1}p}|^2}{(1+u_{1\bar{1}})^2}  \geq -2(\varphi')^2\sum_{p=1}^nS^{p\bar{p}}\big|\varphi'u_pu_{\bar{p}p}+\varphi'\sum_{j=1}^nu_{jp}u_{\bar{p}}\big|^2-2(\psi')^2\mathcal S|\n u|^2.
\]

\noindent Note that, by the choice of $\varphi$ (\ref{varphi12}), the first term above annihilates the fourth term in (\ref{newsecondorder}). The second one is bounded, using (\ref{proppsi}), by $-2(6C_0+3)^2K\mathcal S$. Furthermore the first term in (\ref{newsecondorder}) is nonnegative by the concavity of the $S= \log \sigma_k$ operator, and the sixth term is also nonnegative by (\ref{propphi}) and (\ref{proppsi}). Coupling the above inequalities we obtain
\begin{eqnarray}\label{whatremains}
0&\geq& \varphi'\sum_{p=1}^nS^{p\bar{p}}\la_p^2-18(2C_0+1)^2K\mathcal S+\frac{(\log f)_{1\bar{1}}}{\la_1}+2\varphi' {\rm Re}[(\log f)_{\bar{j}}u_j]\\\nonumber
&&-(2\varphi'+\psi'-C_0 )k.
\end{eqnarray}
As $\varphi'\geq\frac1{4K}$, the first of these new terms is estimable by 
\[
\varphi'\sum_{p=1}^nS^{p\bar{p}}\la_p^2\geq {1\over 4K} S^{n\bar n} \lambda_n^2 \geq \frac{1}{4nK}\mathcal S\la_n^2\geq \frac{1}{4nK}\delta^2\mathcal S\la_1^2.
\]
Here we used the case assumption and the fact that the coefficients $S^{j\bar{j}}$ increase in $j$. Next, using Corollary \ref{blocki} and the fact that $\|f^{1/(k-1)}\|_{\mathcal C^1}, \|f^{1/(k-1)}\|_{\mathcal C^{1,1}}$ are
bounded, the last three terms in (\ref{whatremains}) can be estimated from below as 
\[
\frac{(\log f)_{1\bar{1}}}{\la_1}+2\varphi' {\rm Re}[(\log f)_{\bar{j}}u_j]-(2\varphi'+\psi'-C_0 )k
\geq -\frac{C}{\la_1f^{1/(k-1)}}-\frac{C}{\sqrt{K}f^{1/(k-1)}}-C
\]
for some constant $C$ dependent  on $C_0, k, B$ and $n$. Finally by MacLaurin inequality
\begin{eqnarray}\label{use-garding}
\mathcal S&=&\sum_{p=1}^nS^{p\bar{p}}=\sum_{p=1}^n\frac{\sigma_{k-1}(\la|p)}{\sigma_{k}(\la)}=(n-k+1)\frac{\sigma_{k-1}(\la)}{\sigma_{k}(\la)}\\\nonumber
&\geq& c(n,k)\frac{\sigma_k^{(k-2)/(k-1)}\sigma_1^{1/(k-1)}}{\sigma_k}\geq c(n,k)\frac{\la_1^{1/(k-1)}}{f^{1/(k-1)}}.
\end{eqnarray}

\noindent Therefore, multiplying both sides of the inequality (\ref{whatremains}) by $f^{1/(k-1)}$, we get
\[
0\geq c(n,k)\la_1^{1/(k-1)}\left(\frac{\delta^2}{4Kn}\la_1^2-18(2C_0+1)^2K\right)-C-C\left(1+\sup_Mf^{1/(k-1)}\right).
\]
It follows that
\[
\la_1^2\leq CK^2+\frac{CK}{\delta^2\la_1^{1/(k-1)}},
\]
for some $C$ under control.

\smallskip

{\bf Case 2: Assume $\la_n\geq -\delta \la_1$}. Exactly as in \cite{HMW}, we consider
\[
I:=\left\{i\in\{1,\cdots,n\}\, |\,  \sigma_{k-1}(\la|i)>\delta^{-1}\sigma_{k-1}(\la|1)\right\}.
\]

\noindent As $\delta^{-1}={2A+1}\geq 7$, $i=1$ does not belong to $I$. Returning to our setting we get that $p\in I$ if and only if
\[
S^{p\bar{p}}>\delta^{-1}S^{1\bar{1}}.
\]

\noindent Then exploiting (\ref{firstorder}) and the Schwarz inequality we have
\begin{eqnarray*}
&&-\sum_{p\in\{1,\cdots,n\}\setminus I}\frac{S^{p\bar{p}}|u_{1\bar{1}p}|^2}{(1+u_{1\bar{1}})^2}\\
&\geq& -2(\varphi')^2\sum_{p\in\{1,\cdots,n\}\setminus I}S^{p\bar{p}}|u_pu_{p\bar{p}}+\sum_{j=1}^nu_{jp}u_{\bar{p}}|^2-2(\psi')^2\sum_{p\in\{1,\cdots,n\}\setminus I}S^{p\bar{p}}|u_p|^2\\
&\geq&-2(\varphi')^2\sum_{p\in\{1,\cdots,n\}\setminus I}S^{p\bar{p}}|u_pu_{p\bar{p}}+\sum_{j=1}^nu_{jp}u_{\bar{p}}|^2 -18(2C_0+1)^2KS^{1\bar{1}}.
\end{eqnarray*}

\noindent Using the same strategy as in case 1, the first term annihilates the following term in (\ref{newsecondorder}):
\[
\varphi''\sum_{p\in\{1,\cdots,n\}\setminus I}S^{p\bar{p}}|u_pu_{p\bar{p}}+\sum_{j=1}^nu_{jp}u_{\bar{p}}|^2.
\]

\noindent What remains from (\ref{newsecondorder}) can be written as
\begin{eqnarray*}
0&\geq &-\sum_{i,j,r,q=1}^n\frac{S^{i\bar{j},r\bar{q}}u_{i\bar{j}1}u_{r\bar{q}\bar{1}}}{1+u_{1\bar{1}}}-\sum_{p\in I}\frac{S^{p\bar{p}}|u_{1\bar{1}p}|^2}{(1+u_{1\bar{1}})^2}+(-\psi'+\varphi'-2C_0)\mathcal S\\
&&+\varphi''\sum_{p\in I}S^{p\bar{p}}|u_pu_{p\bar{p}}+\sum_{j=1}^nu_{jp}u_{\bar{p}}|^2+\psi''\sum_{p=1}^nS^{p\bar{p}}|u_p|^2+\frac1{4K}\sum_{p=1}^nS^{p\bar{p}}\la_p^2\\
&&+\frac{(\log f)_{1\bar{1}}}{\la_1}+2\varphi' {\rm Re}[(\log f)_{\bar{j}}u_j]-(2\varphi'+\psi'-C_0 )k-18(2C_0+1)^2KS^{1\bar{1}}.
\end{eqnarray*}
If $\la_{1}^2\geq [12(2C_0+1)K]^2$ (which we can safely assume a priori for otherwise we are through) the last term can be absorbed by the sixth one. Also $-\psi'+\varphi'-2C_0\geq 1$, therefore the previous estimate is reduced to 
\begin{eqnarray*}
0&\geq& -\sum_{i,j,r,q=1}^n\frac{S^{i\bar{j},r\bar{q}}u_{i\bar{j}1}u_{r\bar{q}\bar{1}}}{1+u_{1\bar{1}}}-\sum_{p\in I}\frac{S^{p\bar{p}}|u_{1\bar{1}p}|^2}{(1+u_{1\bar{1}})^2}+\varphi''\sum_{p\in I}S^{p\bar{p}}|u_pu_{p\bar{p}}+\sum_{j=1}^nu_{jp}u_{\bar{p}}|^2\\
&&+\psi''\sum_{p=1}^nS^{p\bar{p}}|u_p|^2+\frac1{8K}\sum_{p=1}^nS^{p\bar{p}}\la_p^2+\mathcal S+\frac{(\log f)_{1\bar{1}}}{\la_1}+2\varphi' {\rm Re}[(\log f)_{\bar{j}}u_j]\\
&&-(2\varphi'+\psi'-C_0 )k.
\end{eqnarray*}
As as case 1, the last three terms can be estimated by $-\frac{C}{f^{1/(k-1)}}$
for some constant $C$ dependent on $B, n, C_0$ and $k$. So if the first four terms add up to something nonnegative then we end up with
\[
\frac{C_3}{f^{1/(k-1)}}\geq \mathcal S+\frac1{8K}\sum_{p=1}^nS^{p\bar{p}}\la_p^2.
\]
This together with (\ref{use-garding}) imply $\la_1\leq C$.

\smallskip

What remains is to prove the non-negativity of 
\begin{eqnarray}\nonumber
-\sum_{i,j,r,q=1}^n\frac{S^{i\bar{j},r\bar{q}}u_{i\bar{j}1}u_{r\bar{q}\bar{1}}}{1+u_{1\bar{1}}}-\sum_{p\in I}\frac{S^{p\bar{p}}|u_{1\bar{1}p}|^2}{(1+u_{1\bar{1}})^2}+\varphi''\sum_{p\in I}S^{p\bar{p}}|u_pu_{p\bar{p}}+\sum_{j=1}^nu_{jp}u_{\bar{p}}|^2+\psi''\sum_{p=1}^nS^{p\bar{p}}|u_p|^2.
\end{eqnarray}

\noindent Exploiting (\ref{firstorder}) and Proposition 2.3 from \cite{HMW} the last two terms can be estimated from below by
\[
\sum_{p\in I}S^{p\bar{p}}\left[2(\varphi')^2\big|u_pu_{p\bar{p}}+\sum_{j=1}^nu_{jp}u_{\bar{p}}\big|^2+\frac{2\delta}{1-\delta}|u_p|^2\right]\geq2\delta\sum_{p\in I}\frac{S^{p\bar{p}}|u_{1\bar{1}p}|^2}{(1+u_{1\bar{1}})^2}.
\]

\noindent On the other hand, the concavity of the $S=\log \sigma_k$ operator yields that the first term is  controlled from below by
\[
-\sum_{i,j,r,q=1}^n\frac{S^{i\bar{j},r\bar{q}}u_{i\bar{j}1}u_{r\bar{q}\bar{1}}}{1+u_{1\bar{1}}}\geq -\sum_{p\in I}\frac{S^{p\bar{1},1\bar{p}}|u_{1\bar{1}p}|^2}{1+u_{1\bar{1}}}.
\]

\noindent Therefore, the inequality to-be-proven will be satisfied if
\[
-S^{p\bar{1},1\bar{p}}\geq (1-2\delta)\frac{S^{p\bar{p}}}{\la_1},
\]
for each $p\in I$. Exploiting the formulas for $S^{p\bar{1},1\bar{p}}$ and $S^{p\bar{p}}$ this is in turn equivalent to
\[
\frac{\sigma_{k-2}(\la|1p)}{\sigma_k}\geq (1-2\delta)\frac{\sigma_{k-1}(\la|p)}{\la_1\sigma_k}.
\]
But the inequality above is exactly the inequality proven in \cite{HMW} (page 559).
Indeed, the inequality can be rewritten as
\[
(2\delta\la_1+(1-2\delta)\la_p)\sigma_{k-1}(\la|p)\geq \la_1\sigma_{k-1}(\la|1)
\]
and the latter one holds due to the case assumptions $\la_p\geq\la_n>-\delta\la_1$ and $\sigma_{k-1}(\la|p)\geq\delta^{-1}\sigma_{k-1}(\la|p)$.
 Thus the claimed inequality is proven.

\end{proof}

\section{Examples} 

 In this section we shall investigate the examples from Proposition \ref{Plis} in the real and complex domains. We will also deal with the complex compact manifold case. 
 
 \smallskip

\noindent As mentioned in the Introduction, it was stated in \cite{ITW} that a modification of the argument from the Monge-Amp\`ere case (see \cite{W1}) shows that the exponent $1/(k-1)$
on the right hand side cannot be improved any further. An important feature of these examples is that they are separately radial in all but one of the coordinates and radial in the distinguished coordinate. In the convex setting this means that 
\[
u(x',x_n)=u(y',y_n), \ \textit{ \rm whenever }  |x'|=|y'| \textit{ \rm and } |x_n|=|y_n|. 
\]
Here, we use the notation $x=(x',x_n)=(x_1,\ldots,x_n)$.

\smallskip

Note that, for convex $u$, this implies that $u$ {\it is increasing} in the both radial directions. The same observation holds for a pluri-subharmonic function $v(z',z_n)$ radial in both directions and this was heavily used in \cite{P}. What makes the $k$-Hessian case different is that a priori such a $k$-convex function will be increasing in the directions $x'$ {\it only} as the $2$-convex example
\[
u(x',x_3)=u(x_1,x_2,x_3)=3(x_1^2+x_2^2)-x_3^2
\]
shows. We will nevertheless prove an additional lemma showing that our examples are indeed increasing in the 
radial $x_n$ (respectively $z_n$) directions. Given this lemma, the proof is indeed analogous to the one in \cite{W1} in the $k$-convex case and to \cite{P} in the $k$-subharmonic case. Our lemma can also be generalized to work on $\mathbb P^{n-1}\times\mathbb P^1$ equipped with the Fubini-Study product metric and thus provides examples in the case of compact K\"ahler manifolds. Below we provide the full details.

\subsection{Examples in the real setting}
 
In this subsection we fix $1<k\leq n$. We will work in the unit ball $B=B(0,1)$ in $\mathbb R^n$. The following lemma is crucial for our construction.

\begin{lemma}\label{minimum}
Let  $u$ be a continuous $k$-convex function on $B$ which is constant on $\pa B$ and it depends only on $|x_n|$ and $|x'|$.
Assume that  $F=S_k(D^2u)$ is (weakly) decreasing with respect to $x_n$. Then  $u$ is weakly increasing with respect to $|x_n|$.
In particular,
\begin{equation}\label{min}
\inf_Bu=u(0).
\end{equation}
\end{lemma}

\begin{proof}
 Observe that for for $k=n$ this follows simply from the convexity of $u$, but for $k<n$ we have to work harder. 
 
 Note that $u$ is radially invariant in the $x_n$ direction, it suffices to prove that for each $x'\in\mathbb R^{n-1}, |x'|<1$ the function $t\rightarrow u(x',t)$ is increasing on the interval $(0,\sqrt{1-|x'|^2})$. For any $\varepsilon,\delta>0$, we define 
 \[
 v_\varepsilon(x', x_n)=u(x', x_n)+2\varepsilon|x|^2
 \]
and
\[
w_{\varepsilon,\delta}(x',x_n)=u(x',x_n+\delta)+4\varepsilon.
\]
Then, our goal is to show that $v_{\varepsilon}(x', x_n) \leq w_{\varepsilon,\delta}(x',x_n)$ for any $(x', x_n)\in B$ such that $(x', x_n + \delta) \in B$ and $x_n>0$. If this holds, taking $\varepsilon \rightarrow 0$, we obtain 
\[
u(x', x_n) \leq u(x', x_n+ \delta)
\]
for any $\delta>0$ and $0<x_n< x_n + \delta < \sqrt{1- |x'|^2}$.

\smallskip

To obtain the desired inequality, we first observe that, using the assumption that $S_k(D^2u)$ is (weakly) decreasing with respect to $x_n$, we have
\begin{equation}\label{h43}
S_k(D^2(v_\varepsilon-\varepsilon|x|^2))=S_k (D^2 u(x', x_n) + 2\varepsilon I_n) > S_k (D^2 u(x', x_n)) \geq S_k(D^2 w_{\varepsilon,\delta}) 
\end{equation}
 on $S_\delta=\{(x',x_n):x_n>-\delta/2,(x',x_n+\delta)\in B, (x',x_n)\in B\}.$

\smallskip

Moreover, we note that for any $\varepsilon>0$ one has 
\begin{equation}\label{delta0}
\lim_{\delta\rightarrow 0^+}w_{\varepsilon,\delta}(x)=u(x)+4\varepsilon>v_{\varepsilon}.
\end{equation}
Now, we fix a small $\varepsilon>0$ and suppose that there is a $\delta$ such that 
\[
v_\varepsilon(p)\geq w_{\varepsilon,\delta}(p)
\]
for some point $p\in \bar S_{\delta}\cap\{x_n\geq 0\}$. Because of (\ref{delta0}), we know there must exist the smallest such $\delta$, and we denote it by $\delta_0$. Then, there is a $p_0\in \bar S_{\delta_0}\cap\{x_n\geq 0\}$ such that $v_\varepsilon(p_0)\geq w_{\varepsilon,\delta_0}(p_0)$.

\noindent For $(x', x_n) \in \pa S_{\delta_0}\cap\{x_n\geq 0\}$ (in fact even on $\pa S_{\delta_0}\cap\{x_n>-\delta_0/2\}$), we know $(x', x_n + \delta_0) \in \pa B$. Therefore, 
\[
w_{\varepsilon, \delta_0}(x', x_n) = u(x', x_n+\delta_0) + 4\varepsilon = \max_B \, u + 4\varepsilon.
\]
Here we used that fact that $u$ is $k$-convex and equals to constant on $\pa B$ and hence $u$ attains its maximum on $\pa B$. Recalling the definition of $v_{\varepsilon}$, we have
\begin{equation}\label{boundary1}
w_{\varepsilon, \delta_0}(x) > v_{\varepsilon}(x),  \ \textit{ \rm for } x\in \pa S_{\delta_0}\cap\{x_n\geq 0\}.
\end{equation}
In particular, this implies $p_0 \in S_{\delta_0}\cap\{x_n\geq 0\}$.

\smallskip

On the other hand, for any point $q=(q',q_n)\in S_{\delta_0}\setminus\{x_n\geq 0\}$. We compute 
\begin{equation}\label{negative}
w_{\varepsilon,\delta_0}(q)=w_{\varepsilon,\delta_0+2q_n}(q',-q_n)>v_\varepsilon(q',-q_n)=v_\varepsilon(q),
\end{equation}
where we used the minimality of $\delta_0$ and  negativity of $q_n>-\delta_0/2$. Therefore, we have
\begin{equation}\label{boundary2}
w_{\varepsilon, \delta_0}(x) \geq v_{\varepsilon}(x),  \ \textit{ \rm for } x\in S_{\delta_0}\cap\{x_n=0\}.
\end{equation}
By (\ref{h43}) and the maximum principle applied to the domain $ S_{\delta_0}\cap\{x_n> 0\}$ we obtain
\[
w_{\varepsilon,\delta_0}\geq v_\varepsilon\; \hbox{ on } \; S_{\delta_0}\cap\{x_n> 0\}.
\] 
This together with (\ref{negative}) implies 
\begin{equation}\label{geq}
w_{\varepsilon,\delta_0}(x)\geq v_\varepsilon (x) \ \ \textit{ \rm for } \forall x \in S_{\delta_0}.
\end{equation}

\smallskip

Next, we want to show that the above inequality is indeed a strict inequality. For this, we consider the open set
\[
D=\left\{x\in S_{\delta_0}:v_\varepsilon(x)-\varepsilon\, |x-p_0|^2+\varepsilon\, {\rm dist}(p_0,\pa S_{\delta_0})^2/4>w_{\varepsilon,\delta_0}(x)\right\}.
\]
Then, $D\subset\subset S_{\delta_0}$ and it holds
\[
v_\varepsilon(x)-\varepsilon\,|x-p_0|^2+\varepsilon\,{\rm dist}(p_0,\pa S_{\delta_0})^2/4>w_{\varepsilon,\delta_0}(x)\hbox{ in }D
\]
and
\[
v_\varepsilon(x)-\varepsilon\,|x-p_0|^2+\varepsilon\,{\rm dist}(p_0,\pa S_{\delta_0})^2/4=w_{\varepsilon,\delta_0}(x)\hbox{ on } \pa D.
\]
On the other hand, we also have, on $D$,
\[
S_k(D^2(v_\varepsilon(x)-\varepsilon\,|x-p_0|^2+\varepsilon\,{\rm dist}(p_0,\pa S_{\delta_0})^2/4))=S_k(D^2(v_\varepsilon-\varepsilon\,|x|^2))>S_k(D^2w_{\varepsilon,\delta_0}).
\]
Then, we have contradiction by using the comparison principle (Theorem \ref{compprin}). Therefore, we must have
\begin{equation}\label{sgeq}
w_{\varepsilon,\delta_0}(x)> v_\varepsilon (x) \ \ \textit{ \rm for } \forall x \in S_{\delta_0}.
\end{equation}
This contradicts with the choice of $\delta_0$. Therefore, we obtain 
\[
v_\varepsilon(x) < w_{\varepsilon,\delta}(x) \ \ \textit{\rm on } S_{\delta}\cap\{x_n\geq 0\}.
\]
for any $\delta>0$. Letting $\varepsilon \rightarrow 0$, we get that $u$ is weakly increasing in $|x_n|$.
(\ref{min}) follows from the sub-harmonicity of $u$ with respect to $x'$. 
\end{proof}

\medskip

Given Lemma \ref{minimum}, the construction of the example and its justification follow closely the argument in \cite{W1}. We provide the details for the sake of completeness. Let
\begin{eqnarray}\label{defeta}
\eta(t) = \begin{cases}
\exp \left(-1/(1-t^2) \right) \ \ &t<1\\
0 & t\geq 1
\end{cases}
\end{eqnarray}
For $a,b\in\mathbb{R}, a>1$ define 
\begin{equation}\label{defF}
F(x)=\eta\left({|x_n|\over |x'|^a}\right)\, |x'|^b.
\end{equation}

\medskip

\begin{example}\label{Wang}
If $0<b< 2(k-1)(a-1)$, then the k-convex solution $u$ of the Dirichlet problem
\begin{equation}\label{KrylovHMM}
 \begin{cases}
  S_k(D^2u)=F(x) \  &\hbox{ in } B\\
  u=0 & \hbox{ on } \pa B
 \end{cases}
\end{equation} 
is not $\mathcal{C}^{1,1}$ in any neighbourhood of $0$. Furthermore, $F^\gamma\in\mathcal{C}^{1,1}(B)$ for $\gamma>\frac{1}{k-1}+\frac{1}{(a-1)(k-1)}$. In particular, taking $a\rightarrow\infty$ we obtain that no exponent larger than $1/(k-1)$ could yield $\mathcal C^{1,1}$ solutions in general.
  \end{example}
\begin{proof}

The comparison principle implies that  the solution is unique. Because of the rotational invariance of the data the solution has to 
depend only on $|x'|$ and $|x_n|$, i.e. it has to be radial both in the $x'$ and the $x_n$ direction. By Lemma \ref{minimum} it is increasing 
separately in $|x'|$ and in $|x_n|$. 

Let $\varepsilon>0$ be such that $\varepsilon^2+\varepsilon^{2/a}<1$. 
Define the domain $$P=\{(x',x_n):|x'|<\varepsilon^{1/a},|x_n|<\varepsilon\}$$
and the function 
\[
v(x)=\left(\frac{x_1-\frac{1}{2}\varepsilon^{1/a}}{\frac{1}{4}\varepsilon^{1/a}}\right)^2
+\sum_{k=2}^{n-1}\left(\frac{x_k}{\frac{1}{2}\varepsilon^{1/a}}\right)^2
+\left(\frac{x_n}{\frac{1}{4^{a+1}}\varepsilon}\right)^2-1.
\] 
By computation we have
\[
E:=\left\{x\in B:v<0\right\}\subset P.
\]
On the other hand
\[
\inf_EF\geq\eta(1/4)4^{-b}\varepsilon^{b/a}.
\]
Observe also that for some positive constant $c_1$ (independent on $\varepsilon$) the following inequality holds 
\[
S_k\left(D^2v\right)\leq c_1 \varepsilon^{-2-2(k-1)/a}.
\]
Then it is possible to choose another constant $c_2$ (also independent on $\varepsilon$) such that 
\[
S_k(D^2(c_2\,\varepsilon^{\frac{2a+2(k-1)+b}{ka}}\, v+\sup_Pu))\leq \inf_EF.
\]
By the comparison principle 
\[
c_2\, \varepsilon^{2a+2(k-1)+b\over ka}\, v+\sup_Pu\geq u,\  \ \textit{\rm  on } P 
\]
and we obtain 
\[
u(0)\leq u(\frac{1}{2}\varepsilon^{1/a},0,\ldots,0)\leq \sup_P u-c_2\,\varepsilon^{\frac{2a+2(k-1)+b}{ka}}=u(\varepsilon^{1/a},0,\ldots,0,\varepsilon)-c_2\,\varepsilon^{\frac{2a+2(k-1)+b}{ka}}.
\]
For the last equality, we used the fact that $u$ obtains its maximum on $\pa P$ since $u$ is radial and increasing in both $x'$
and $x_n$ directions.
Denote 
\[
s(t)=u(0,t), \ \  \textit{ \rm and }\ \ell(t)=u(t^{1/a},0,\ldots,0,t) 
\]
for $t\in[0,1]$. We clearly have $s\leq \ell$ since $u$ is increasing in the $|x'|$ direction.
 
 \smallskip
 Assume that $s<\ell$ on some interval  $(c,d)$. Then, for any $\varepsilon_1,\varepsilon_2\in(c,d)$ with $\varepsilon_2-\varepsilon_1>0$ small enough, we can find an affine function $w$ 
 dependent only on $x_n$, such that $w(x', t) < \ell(t)$ for any $t\in (\varepsilon_1, \varepsilon_2)$ and
\[
w(0', \varepsilon_1)= u(0', \varepsilon_1)= s(\varepsilon_1), \ \ w(0', \varepsilon_2)= u(0', \varepsilon_2)= s(\varepsilon_2).
\]
 Then, by the monotonicity of $u$ in the $|x'|$ direction again, we have 
 \[
 w(x', x_n) \leq u(x', x_n)  \ \textit{ \rm on } \pa(\{F=0\}\cap\{x_n\in(\varepsilon_1,\varepsilon_2)\}).
 \]
On the other hand, by the definition of $F$, we have $S_k(D^2u)=0=S_k(D^2 w)$ on $\{F=0\}\cap\{x_n\in(\varepsilon_1,\varepsilon_2)\}$. Then, the comparison principle implies
\[
w(x', x_n) \leq u(x', x_n)  \ \textit{ \rm in } \{F=0\}\cap\{x_n\in(\varepsilon_1,\varepsilon_2)\}
 \]
for any small interval $(\varepsilon_1,\varepsilon_2)\in(c,d)$. In particular, $w(0', t) \leq u(0', t) = s(t)$ for any $t\in (\varepsilon_1,\varepsilon_2)$. Thus $s(t)$ is weakly concave on $(c, d)$.

\smallskip

Now assume that $u$ is $\mathcal{C}^{1,1}$ in a neighbourhood of $0$. Then $s'(0)=0$. We claim that $s$ cannot be concave in any interval of the type $(0,r)$. Indeed, if $s(t)$ is weakly concave on some interval $(0, r)$, then it follows that $s'(t)\leq 0$ for $t\in (0, r)$. On the other hand, by Lemma \ref{minimum}, we have $s'(t)\geq 0$. Therefore, $s'(t)\equiv 0$ and hence $s(t)$ is constant on $(0, r)$. Taking largest such $r$ (which is strictly less than one for otherwise the function would be globally constant), we have $s(r)<\ell(r)$ as $u$ is not constant in a neighbourhood of zero. But then applying the above argument around the point $r$, we would obtain that $s$ is concave at $r$. This contradicts with the fact that $s$ is constant to the left of $r$ and strictly increases to the right of $r$. This proves the claim.

Then, it follows that the strict inequality $s(t) < \ell(t)$ can not hold in any interval of the type $(0,r)$. Thus there is a sequence $\varepsilon_m$ decreasing to $0$ such that 
\[
u(0,\varepsilon_m)=s(\varepsilon_m)=l(\varepsilon_m)\geq u(0)+c_2\,\varepsilon_m^{\frac{2a+2(k-1)+b}{ka}}
\]
and we can conclude
\[
u(0,\varepsilon_m)-u(0)\geq c_2\,\varepsilon_m^{2-\theta/ak},
\]
where $\theta=2(k-1)(a-1)-b$. This contradicts the assumption that $u\in\mathcal{C}^{1,1}$ around $0$.
\end{proof}

\subsection{Compact K\"ahler manifold case}

Now we deal with the compact K\"ahler manifold case. The construction is similar to the real case and the main technical difficulty is that we have to replace the translation operators with suitable automorphisms of the K\"ahler manifold. These automorphisms will furthermore preserve the K\"ahler form.

\smallskip
\noindent
We fix $1<k\leq n$ in what follows. 
The examples  will be constructed on $\mathbb{P}^{n-1}\times\mathbb{P}^1$ equipped with the product metric $\omega=\omega_{FS}'+\omega_{FS}$ with $\omega_{FS}$ denoting the 
Fubini-Study metrics on each factor. For $z\in\C$ we split the coordinates and write $z=(z',z_n)\in\mathbb{C}^{n-1}\times\mathbb{C}$
which we identify in the usual way as a subset (the affine chart) of $\mathbb{P}^{n-1}\times\mathbb{P}^1$. On this affine chart we have $\omega_{FS}'=i\ddb \left(\frac{1}{2}\log(1+|z'|^2)\right)$ and $\omega_{FS}=i\ddb\left(\frac{1}{2}\log(1+|z_n|^2)\right)$. The following complex analogue of Lemma \ref{minimum} is crucial for the construction.

\begin{lemma}\label{monotonicznosc}
 Let  $\varphi\in\ksh(\mathbb{P}^{n-1}\times\mathbb{P}^1, \om)$ be a continuous function such that $(\om+i\ddb\varphi)^k\we\om^{n-k}=f\, \om^n$. Moreover, assume that
\begin{enumerate}
\item[1).] for any $r>0$ the set $\{(z',z_n)\in\mathbb{C}^{n-1}\times\mathbb{C}: |z_n|\leq r,\ f(z',z_n)=0\}$ is bounded;
\item[2).] $\varphi|_{\C}$ (and hence $f$) depends only on $|z'|$ and $|z_n|$ on the affine chart;
\item[3).] $f(z', z_n)$ is strictly decreasing in $|z_n|$ for all fixed $z'$ such that $f(z',z_n)>0$.
\end{enumerate}
 Then the function  $\varphi$  is increasing with respect to   $|z_n|$. 
\end{lemma}

\begin{proof}
  Denote by $t_\alpha$ and $T_\alpha$ the automorphisms 
 of  $\mathbb{P}^1$ and $\mathbb{P}^{n-1}\times\mathbb{P}^1$ respectively given by
\[
t_{\alpha}\left([w_0:w_1]\right)=[\,\cos(\alpha)w_0-\sin(\alpha)w_1\,:\,\sin(\alpha)w_0+\cos(\alpha)w_1\,];
\] 
\[
T_\alpha\left([z_0:\cdots z_{n-1}]\times[w_0:w_1]\right)=[z_0:\cdots:z_{n-1}]\times t_{\alpha}\left([w_0:w_1]\right).
\]
We would like to point out that $t_{\alpha}$ preserves $\om_{FS}$ while $T_{\alpha}$ preserves the product metric $\om$. Moreover, on the affine chart of $\mathbb P^1$, $t_{\alpha}$ reads
\[
t_\alpha(z)=\frac{z+\tan\alpha}{1-z\tan\alpha}.
\] 
Choose now $\varepsilon>0$ and fix an angle $\alpha\in(0,\frac{\pi}{4}]$.
 Let $W=\{z\in\mathbb{C}:{\rm Re}\, z\geq0\}\cup\{\infty\}\subset\mathbb{P}^1$ and $ E=T_{\alpha\over 2}^{-1}(\mathbb{P}^{n-1}\times W)$. For $(z',z_n)\in {\rm int}(E)$ 
 we have 
 \begin{equation}\label{Ealpha}
 t_\alpha(z_n)=\infty\ \mbox{ or } \ |z_n|<|t_\alpha(z_n)|.
 \end{equation}
Let $\psi:\mathbb{P}^{n-1}\times\mathbb{P}^1\rightarrow\mathbb{R}$ be a continuous function given by
 \[
 \psi(z',z_n)=(\varphi\circ T_\alpha)(z',z_n)+\varepsilon.
 \]
 For $z\in\partial  E$, we have $|z_n|= |t_{\alpha}(z_n)|$ and hence $\varphi(z)=\varphi\left(T_{\alpha}(z)\right)<\psi(z)$.
 Thus, for any $\delta>0$ small enough, the set 
 \[
 D:=\{\varphi-\delta>\psi\}\cap E
 \] is relatively compact in ${\rm int} (E)$. 
 The monotonicity properties of $f$ imply that
\begin{equation*}
f(z)\geq f(T_{\alpha}(z)) \mbox{ for } z\in\overline{E}.
\end{equation*}
The comparison principle (\ref{compprinkahler}) results in 
\[
\int_D f\om^{n}\geq\int_D f\circ T_{\alpha}\,\om^{n}=\int_D (\om+i\ddb\psi)^k\we\om^{n-k}\geq\int_D (\om+i\ddb\varphi)^k\we\om^{n-k}=\int_D f\om^{n}.
\]
Together with assumption (3) and (\ref{Ealpha}), this gives us $f=0$ on $D$. We wish to point out that, contrary to the local setting, we cannot deduct
the emptiness of $D$ at this stage since we do not know whether $D$ is contained in some affine chart. To this end we use assumption (1). Note that the
projection of $E$ onto the $\mathbb P^1$ factor is a bounded subset of the affine chart. 
By assumption (1) we get that $D$ is bounded. Then the comparison principle for bounded domains implies that $D$ is empty. 

\end{proof}

\medskip


Let $\eta$ be as in the real case.
For $a\geq1,b\in\mathbb{R}$ and $z\in\C$, define
\begin{eqnarray}\label{define-f}
f(z)=A\, \exp(-|z|^2)\,\eta\left({|z_n|\over |z'|^a}\right)|z'|^b
\end{eqnarray}
and extend $f$ by zero on the divisors of infinity so that $f$ is a function on $\mathbb P^{n-1}\times\mathbb P^1$. Here, $\eta(t)$ is given in (\ref{defeta}) and the constant $A>0$ is chosen such that 
\begin{equation}\label{normalization}
\int_{\mathbb{P}^{n-1}\times\mathbb{P}^1}f\omega^n=\int_{\mathbb{P}^{n-1}\times\mathbb{P}^1}\omega^n.
\end{equation}
\begin{lemma}\label{inequality}
If $\varphi\in\ksh(\mathbb{P}^{n-1}\times\mathbb{P}^1,\om)$ is the unique continuous solution to the equation
\begin{equation}\label{DPcn}
 (\omega+i\ddb\varphi)^k\wedge\om^{n-k}=f\omega^n 
\end{equation} 
on $\mathbb{P}^{n-1}\times\mathbb{P}^1$, satisfying $\varphi(0)=0$. Then there exist a constant $c>0$ and a sequence $\varepsilon_m>0$ which decreases to $0$, such that
\begin{equation}\label{nierreg}
u(0,\varepsilon_m)\geq c\,\varepsilon_m^\theta,
\end{equation} 
where $\theta=\frac{2a+2k-2+b}{ka}$ and $u=\varphi|_{\C}+\frac{1}{2}\log(1+|z'|^2)+{1\over 2}\log(1+|z_n|^2)$.
  \end{lemma}
\begin{proof}
We remark that the solution is unique (uniqueness for normalized solutions to complex Hessian equations holds in much greater generality that we need here, see for example \cite{DC}). Just as in the real case this implies that it depends only on $|z'|$ and $|z_n|$ and thus by definition $u$ depends only on  $|z'|$ and $|z_n|$ too. By sub-harmonicity $u$ is increasing with respect to $|z'|$ and by Lemma \ref{monotonicznosc} it is strictly increasing in  $|z_n|$ as the function $f$ satisfies all the assumptions in that lemma. 

\smallskip
 From now on, we restrict our attention to the affine chart.
Let $\varepsilon>0$ be such that $\varepsilon^2+\varepsilon^{2/a}\leq1$. 
Let
 \[
 P=\{(z',z_n):|z'|<\varepsilon^{1/a}, |z_n|<\varepsilon\}
 \]
and
 \[
 v(x)=\left(\frac{{\rm Re}\,z_1-\frac{1}{2}\varepsilon^{1/a}}{\frac{1}{4}\varepsilon^{1/a}}\right)^2+\left(\frac{{\rm Im}\,z_1}{\frac{1}{4}\varepsilon^{1/a}}\right)^2+\sum_{k=2}^{n-1}\left(\frac{|z_k|}{\frac{1}{4}\varepsilon^{1/a}}\right)^2+\left(\frac{|z_n|}{\frac{1}{4^{a+1}}\varepsilon}\right)^2.
 \]
  Then 
\[
E=\left\{x\in B:v<1\right\}\subset P.
\]
We have $$\inf_Ef\geq \exp(-2)\,\eta(1/4)\,4^{-b}\,\varepsilon^{b/a}$$
and we can choose an absolute constant $c_1$  such that 
\[
(i\ddb v)^k\wedge\om^{n-k}\leq c_1 \varepsilon^{-2-2(k-1)/a}\om^n \mbox{ on } P.
\]
Then it is possible to choose a constant $c_2$ independent on $\varepsilon$ such that 
\[
\left(i\ddb(\varepsilon^{\frac{2a+2(k-1)+b}{ka}}v)\right)^k\wedge\om^{n-k}\leq c_2f\,\om^n\ \mbox{ on } E.
\]
By the comparison principle  we obtain 
\begin{equation}\label{nierzzporownawczej}
 u(\varepsilon^{1/a},0,\ldots,0,\varepsilon)= \sup_P u\geq\sup_E u\geq\inf_{\partial E}
 c_2\,\varepsilon^{\frac{2a+2(k-1)+b}{ka}}v=c_2\,\varepsilon^{\frac{2a+2(k-1)+b}{ka}}.
\end{equation}
For $t\in\mathbb{R}$, we define
\[
s(t)=u(0,e^t), \ \textit{ \rm and } \ell(t)=u(e^{t/a},0,\ldots,0,e^t).
 \]
 We have $s\leq \ell$. Follow the same argument as in the real case, we can obtain that if $s<\ell$ on some interval $(c, d)$ then $s$ is weakly concave on $(c, d)$. 
However, we also know that $s$ is strictly increasing and
 \[
\lim_{t\to-\infty}s(t)=0,
\]
and this imply $s$ can not be weakly concave in any interval $(-\infty, r)$. Therefore, there is a sequence
 $t_{m}\searrow-\infty$ such that $s(t_m)=\ell(t_m)$. 
Taking $\varepsilon_m=e^{t_m}$ and using (\ref{nierzzporownawczej}) we obtain the Lemma. 
\end{proof}

\medskip


Observe that in the argument above the parameter $b$ can be negative and then the right hand side is merely $L^p$ integrable. In such a case a result
from \cite{DK1} shows that local solutions are bounded for $L^p$ integrable right hand side provided $p>\frac nk$. It is natural to ask the following question:
\begin{question}
Consider the $k$-Hessian equation on a compact K\"ahler manifold $(X, \omega)$
\[
(\omega+ i\ddb u)^k\wedge\omega^{n-k} = f(z) \omega^n
\]
with $0\leq f\in L^p(X)$ satisfying the normalization condition $\int_X f\, \omega^n = \int_X \omega^n$. What is the best possible regularity one can expect for the solution $u$?
\end{question}
We can also ask similar question with the condition on $f$ being replaced by $f\in C^{0, \delta}$ for some $0<\delta <1$. Indeed, by varying the parameters $a$ and $b$ in the example provided in Lemma \ref{inequality}, we have some assertions about what kind of regularity one can expect under different conditions of the right hand side function.

\begin{example}\label{general-example}$\\$
{\rm 
(1). Let $b=-{2a\over p}+2$ with $p>1$. Then $f\in L^p$. (In fact, any $b>2a/p-2(n-1)/p$ yields $L^p$ right hand side). But in (\ref{nierreg})
\[
\theta=\frac{2}{k}\left(1-\frac{1}{p}\right)+2\frac{1}{a}=\frac{2}{k}\left(1-\frac{1}{p}\right)+O\left({1\over a}\right) \mbox{ as } a\rightarrow+\infty.
\]
This shows that we cannot get better than H\"{o}lder regularity for $\varphi$. Moreover, the H\"older exponent can be at most $\frac{2}{k}\left(1-\frac{1}{p}\right)$.

\noindent (2). Similarly, for $b=2$, we have $f\in \mathcal{C}^{0,\delta}$ for some small $\delta>0$ and 
\[
\theta\leq\frac{2}{k}+\frac{2}{a}.
\]
\noindent (3). For $k\geq 3$ and $b=(k-2)(a-2)-3$, we have $f^\gamma\in\mathcal{C}^{0,1}$ for $\gamma=\frac{a}{(k-2)(a-1)-3}\rightarrow{1\over k-2}$ as $ a\rightarrow+\infty$. We can compute
\[
\theta=1-\frac{1}{ak}.
\]
\noindent (4). For $b=2(k-1)(a-2)$, we have $f^\gamma\in\mathcal{C}^{1,1}$ for $\gamma=\frac{1+2/(a-2)}{k-1} \rightarrow {1\over k-1}$  as $a\rightarrow+\infty.$ In this case
\[
\theta=2 - 2\left({1\over a} - {1\over ka}\right).
\]
}
To summarize, by varying the parameter $a$ and $b$, the examples imply
\begin{itemize}
\item For $p>1$, $\gamma>\frac{2}{k}(1-\frac{1}{p})$: $f\in L^p\nRightarrow\varphi\in\mathcal{C}^{0,\gamma}$.
\item For $\gamma>0$ there is $\delta=\delta(n,\gamma)$ such that $f\in \mathcal C^{0,\delta}\nRightarrow\varphi\in\mathcal{C}^{1,\gamma}$.
\item For $k\geq 3$, $\gamma>\frac{2}{k}$ there is $\delta=\delta(n,\gamma)$ such that $f\in \mathcal C^{0,\delta}\nRightarrow\varphi\in\mathcal{C}^{0,\gamma}$.
\item For $k\geq 3$, $a>\frac{1}{k-2}$ there is $\gamma<1$: $f^a\in\mathcal{C}^{0,1}\nRightarrow\varphi\in\mathcal{C}^{0,\gamma}$.
\item For  $a>\frac{1}{k-1}$ there is $\gamma<1$: $f^a\in\mathcal{C}^{1,1}\nRightarrow\varphi\in\mathcal{C}^{1,\gamma}$.
\end{itemize}
\end{example}

\subsection{The case of complex Hessian equations in domains}

Finally we mention that in the case of the complex Hessian equation on domains the following examples can be constructed:
\begin{example}\label{complexWang} Let $a, b\in \mathbb R$ be two numbers satisfying $0<b<2(k-1)(a-1)$.  Consider the Dirichlet problem in the unit ball in $\mathbb C^n$
\begin{equation}\label{complexKrylovHMM}
 \begin{cases}
  (i\ddb u)^k\wedge\beta^{n-k}=F  &\hbox{ in } B \\
  u=0 & \hbox{ on } \pa B,
 \end{cases}
\end{equation} 
where the solution $u$ is assumed to be $k$-subharmonic and $F$ is given in (\ref{defF}). Then $u$
is not $\mathcal{C}^{1,1}$ in any neighbourhood of $0$. But, $F^\gamma\in\mathcal{C}^{1,1}(B)$ for any $\gamma>\frac{1}{k-1}+\frac{1}{(a-1)(k-1)}$. In particular,
no exponent larger than $1/(k-1)$ could yield $\mathcal C^{1,1}$ solutions in general.
  \end{example}
 \begin{proof} The proof repeats the previous cases once one establishes an analogue of Lemma \ref{minimum}.
 We leave the details to the Reader.
 \end{proof}

%
%

\

\bigskip

{\small Institute of Mathematics, Jagiellonian University, ul \L ojasiewicza 6, 30-348 Krak\'ow, Poland}

\smallskip
slawomir.dinew@im.uj.edu.pl

\bigskip
{\small Institute of Mathematics, Cracow University of Tech., Warszawska 24, 31-155 Krak\'{o}w, Poland}

\smallskip
splis@pk.edu.pl

\bigskip

{\small Department of Mathematics, University of California, Irvine, CA 92697, USA}

\smallskip
xiangwen@math.uci.edu

\end{document}